\documentclass[a4paper,11pt]{article}

\usepackage{amsmath}
\usepackage{amssymb}
\usepackage{amsthm}
\usepackage{graphicx}
\usepackage{enumerate}
\usepackage[dvipsnames]{xcolor}
\usepackage{hyperref}
\usepackage{algorithm, algorithmicx, algpseudocode}
\usepackage{subcaption}
\usepackage{comment}
\newtheorem{theorem}{Theorem}[section]

\newtheorem{lemma}[theorem]{Lemma}

\newtheorem{assumption}{Assumption A}
\theoremstyle{definition}

\numberwithin{equation}{section}
\newcommand{\N}{\mathcal{N}}
\newcommand{\D}{\mathcal{D}}

\begin{document}
	
\title{ASPEN: An Additional Sampling Penalty Method for Finite-Sum Optimization Problems  with Nonlinear Equality Constraints}

\author{
Nataša Krejić\thanks{Department of Mathematics and Informatics, Faculty of Sciences, University of Novi Sad, Trg Dositeja Obradovića 4, 21000 Novi Sad, Serbia. Emails: \texttt{natasak@uns.ac.rs}, \texttt{natasa.krklec@dmi.uns.ac.rs}}, 
Nataša Krklec Jerinkić\footnotemark[1], 
Tijana Ostojić\thanks{Department of Fundamental Sciences, Faculty of Technical Sciences, University of Novi Sad, Trg Dositeja Obradovića 6, 21000 Novi Sad, Serbia. Email: \texttt{tijana.ostojic@uns.ac.rs}}, 
Nemanja Vučićević\thanks{Department of Mathematics and Informatics, Faculty of Sciences, University of Kragujevac, Radoja Domanovića 12, 34000 Kragujevac, Serbia. Email: \texttt{nemanja.vucicevic@pmf.kg.ac.rs}.} \footnote{Corresponding author.}
}
    \date{August 4, 2025}
\maketitle

\begin{abstract} We propose a novel algorithm for solving non-convex,  nonlinear equality-constrained finite-sum optimization problems. The proposed algorithm incorporates an additional sampling strategy for sample size update into the well-known framework of quadratic penalty methods. Thus, depending on the problem at hand, the resulting method may exhibit a sample size strategy ranging from a mini-batch on one end,  to increasing sample size that achieves the full sample eventually, on the other end of the spectrum. A non-monotone line search is used for the step size update, while the penalty parameter is also adaptive. The proposed algorithm avoids costly projections, which, together with the sample size update, may yield significant computational cost savings. Also, the proposed method can be viewed as a transition of an additional sampling approach for unconstrained and linear constrained problems,  to a more general class with non-linear constraints.  The almost sure convergence is proved under a standard set of assumptions for this framework, while numerical experiments on both academic and real-data based machine learning problems demonstrate the effectiveness of the proposed approach.  
\end{abstract}

\section{Introduction}\label{sec1}

We consider a nonlinear equality-constrained optimization problem of the form
\begin{align}\label{OriginalProblem}
    \min_x f(x) := \dfrac{1}{N}\sum_{i=1}^{N} f_i(x), \quad \mbox{ subject to } \quad h(x)=0,
\end{align}
where $f_i: \mathbb{R}^n \to \mathbb{R}$, $i=1,\ldots,N$ and $h: \mathbb{R}^n \rightarrow \mathbb{R}^m$ are continuously differentiable.
Optimization problems of this structure naturally arise in various real-world contexts, particularly in areas such as machine learning, deep learning, and network system optimization. They are commonly encountered when training models such as logistic regression and deep neural networks, where the goal is to minimize a cumulative loss across a dataset while satisfying specific constraints. Owing to both their practical relevance and mathematical complexity, these problems continue to attract significant attention in contemporary research.

A wide range of deterministic algorithms has been developed for solving equality-constrained nonlinear optimization problems. One of the frequently used approaches is based on penalty methods and augmented Lagrangian techniques \cite{huyer}. In these methods one reformulates the constrained problem by adding a penalty term to the objective function to penalize constraints violation, governed by a penalty parameter, thereby transforming the problem into an unconstrained one. Penalty methods can also be effectively applied in stochastic environments. In the stochastic case, when either the objective function or the constraints are defined in terms of expectations, several methods are proposed for solving such problems by penalty-based techniques \cite{krejic1, krklecjerinkic1, polak, wang}.

Stochastic approaches that use stochastic approximations of the gradient have been proposed to reduce computational cost per iteration, leading to various stochastic penalty and projection methods  \cite{lan, nandwani, negiar}.

Another popular class of algorithms for the considered problems is based on Sequential Quadratic Programming (SQP). In the stochastic setting, the search direction is typically generated by solving a quadratic subproblem by using a stochastic gradient estimate \cite{berahas2}, and the convergence complexity of the method has been studied in \cite{curtis2}. Various extensions have been proposed, including adaptive sampling techniques \cite{berahas1}, variance reduction \cite{Numerika}, and structural improvements \cite{berahas3, curtis, na1, na2}.  In \cite{fang}, the TR-StoSQP algorithm is introduced, which combines a trust region approach with adaptive trust region radii and allows indefinite Hessians in the subproblem.

In large-scale finite-sum minimization problems, exact evaluation of the objective and its derivatives is computationally expensive. Hence, approximation via subsampling is widely used. Adaptive subsampling - adjusting sample size during iterations - has proven to be an efficient strategy when coupled with line search or trust region mechanisms \cite{iusem1, iusem2, krejic4, krejic2}. Some methods employ additional sampling to decide whether to increase the sample size. Moreover, additional sampling is also used to govern the acceptance of a candidate point, and this approach proved to be very efficient for unconstrained finite-sum minimization \cite{LSOS, krejic4, 23, lsnmbb}.

In \cite{krejic5}, a stochastic first-order method with variable sample size was proposed for minimizing a weighted finite sum under linear equality constraints, using adaptive sampling and approximate projections. The method proposed here generalizes this approach to nonlinear equality constraints, but does not rely on the approximate projection strategy used in \cite{krejic5}.

The non-monotone line search we propose is presented in \cite{lifukushima} and is used in many methods  relax the strict Armijo line search conditions and enhance convergence speed.  In the stochastic framework, non-monotone strategies have also proved their efficiency. Paper \cite{krejic3} proposed a class of algorithms with adaptive sample size combined with non-monotone line search. Non-monotonicity has also been employed in stochastic optimization settings in \cite{krejic4, krejic5, krklecjerinkic4}. 

Solving nonlinear equality-constrained problems with finite-sum structure is challenging due to the high cost of evaluating the objective and its derivatives at each iteration. While penalty methods and SQP-type algorithms have been successfully applied in this setting, most of the existing methods either assume access to full gradients or rely on expensive projection steps. Motivated by these limitations, we designed an efficient first-order algorithm that uses subsampling, avoids expensive projections and the strict decrease imposed by the standard Armijo line search. 

To summarize, our main contributions are as follows:
\begin{itemize}
\item[-] We propose a novel first-order algorithm (ASPEN) for nonlinear equality-constrained finite-sum problems  that combines adaptive sampling, nonmonotone line search, and adaptive penalty parameter update; 
\item[-] ASPEN extends the adaptive sampling framework of \cite{krejic5} to a more general class of problems with nonlinear constraints, while avoiding the (approximate) projections;
\item[-] The theoretical analysis of the proposed method is presented - almost sure convergence of ASPEN is proved   under some standard assumptions for the considered framework;
\item[-] The practical efficiency of ASPEN is demonstrated by comparing it to the relevant state-of-the-art methods  and by illustrating its adaptive nature.  
\end{itemize}

{\bf{Paper organization.}  }{The paper is organized as follows.  In Section \ref{sec3}, we present the proposed algorithm. Section \ref{sec4} contains theoretical analysis and convergence results. Section \ref{sec5} presents numerical experiments. We conclude the paper in Section \ref{sec6} with a summary and directions for future research.}  

\textbf{Notation.}   {Throughout the paper, we use the following notation:
\(\mathbb{R}_+\) denotes the set of non-negative real numbers.
Depending on the argument, the symbol \(\|\cdot\|\) represents the  Euclidean vector norm and the spectral matrix norm. We use “a.s.” to abbreviate “almost sure”/"almost surely". 
$ \mathbb{E}(\cdot) $ and $ \mathbb{E}(\cdot \mid \mathcal{F}) $ denote mathematical expectation and conditional expectation with respect to the \(\sigma\)-algebra \(\mathcal{F}\), respectively.
Finally, for a finite set \(N\), \(|N|\) denotes its cardinality.}  

\section{The algorithm}\label{sec3}

The method that will be proposed within this section  - ASPEN - belongs to a class of first-order methods, i.e., we assume that only the relevant functions and their first-order derivatives are attainable. Given that the full sample size $N$  (the number of data points in machine learning problems) is typically and the data sets often include some data redundancy, we employ a subsampling strategy to reduce the computational load. More precisely,  we use the standard Sample Average Approximations (SAA) to form the objective function approximations. Furthermore, we use the same sample to calculate the gradient approximation, i.e., we take
\begin{align}
f_{\mathcal{N}_k}(x)=\dfrac{1}{N_k}\sum\limits_{i\in\mathcal{N}_k} f_i(x), \,\,\, \nabla f_{\mathcal{N}_k}(x)=\dfrac{1}{N_k}\sum\limits_{i\in\mathcal{N}_k} \nabla f_i(x),
\end{align}
where $\mathcal{N}_k \subseteq \mathcal{N}:=\{1,2,...,N\}$ and  $N_k=|\mathcal{N}_k| $. We emphasize that the choice of sample $\mathcal{N}_k$ is arbitrary, i.e.,  we do not make any restrictions on the sampling strategy for $\mathcal{N}_k$. 

ASPEN fits into the framework of penalty methods. In particular, we use the quadratic penalty function, which takes the following form when considering the  original problem \eqref{OriginalProblem} 
\begin{align*}\label{penalty}
F(x, \mu)=f(x)+\dfrac{\mu}{2}|| h(x)||^2=:f(x)+\mu q(x).
\end{align*}
Here, the parameter $\mu$ is the penalty parameter, and by $q$ we denote the constraints violation function defined by the square norm of $ h$. We assume that the constraints are non-linear in general, but computable under reasonable costs as well as their first-order derivatives. Therefore, we form an approximation of the penalty function as follows
 \begin{equation*}\label{fnk} 
        F_{\mathcal{N}_k}(x, \mu_k)=f_{\N_k}(x)+\dfrac{\mu_k}{2}|| h(x)||^2,
    \end{equation*}
where $\mu_k$ represents the penalty parameter used at iteration $k$. We define  the gradient\footnote{The gradient will always be taken with respect to variable $x$, i.e., $\nabla F_{\mathcal{N}_k}(x, \mu_k):=\nabla_{x} F_{\mathcal{N}_k}(x, \mu_k)$. We drop $x$ from the subscript in order to simplify the notation.} of the approximate (SAA) penalty function at iteration $k$ accordingly, i.e., 
    \begin{equation}\label{gnk} 
        g_k:=\nabla F_{\mathcal{N}_k}(x_k, \mu_k) =\nabla f_{\mathcal{N}_k}(x_k)+\mu_k\nabla^T h(x_k) h(x_k).
    \end{equation}
Since we work with approximate functions, applying a monotone line search may be too restrictive and ineffective. Therefore, we apply a nonmonotone Armijo-type line search similar to one used in \cite{krejic5}, for instance. Setting the search direction to $p_k=-g_k$, we perform the backtracking line search with respect to the approximate penalty function. More precisely, we determine  the step size  $\alpha_k$ that satisfies the following condition
\begin{equation} \label{line9}
        F_{\mathcal{N}_k}(x_k+\alpha_k p_k, \mu_k)\leq F_{\mathcal{N}_k}(x_k, \mu_k)+\eta \alpha_k g_k^Tp_k+\epsilon_k,
        \end{equation}
where  $\{\epsilon_k\}_{k \in \mathbb{N}}$ is assumed to be a sequence of summable positive numbers, i.e., it satisfies   
$$\sum\limits_{k=0}^\infty \epsilon_k<\infty.$$
After the line search, we set $\bar{x}_k=x_k+\alpha_k p_k$ to be the candidate point for the subsequent iterate. 

As mentioned before, depending on the problem, ASPEN may use subsampling during the whole optimization process (i.e., we can have  $N_k <N$ for all $k \in \mathbb{N}$), or it can reach the full sample size (i.e., $N_k=N$ for all $k$ large enough) and switch to the deterministic mode.   We will refer to the first scenario as mini-batch (MB), while the latter one will be referred to as the full sample (FS) scenario. We assume that, if $N_k=N$, then $\N_k=\N$, i.e., we use the whole set of local cost functions $f_i$ (i.e., the whole set of data points). If ASPEN is in the FS mode at iteration $k$ (i.e., if  $N_k=N$), then the candidate point is unconditionally accepted, and we set $x_{k+1}=\bar{x}_k$. Moreover, since the sample size sequence is nondecreasing, the method behaves like a deterministic\footnote{The method still yields a stochastic sequence of iterates since in the FS scenario there exists a finite, but random iteration $\bar{k}$ such that $N_k=N$ for all $k \geq \bar{k}$.} sequential programming penalty method - when a vicinity of a stationary point of the penalty function $F(\cdot, \mu_k)$ is reached, the penalty parameter is increased to $\mu_{k+1}$ and an  approximate solution $\min_{x \in \mathbb{R}^n} F(x, \mu_{k+1})$ is computed.  The novelty of ASPEN lies in the MB phase that we describe in detail as follows.  
If ASPEN is in the  MB phase at iteration $k$ ($N_k<N$), then an additional check is performed to decide whether to accept the candidate point or to reject the step completely. This check is based on additional sampling \cite{LSOS,krejic4,lsnmbb}, which also guides the sample size update. Namely, we  use an independent, arbitrarily  small subsample $\mathcal{D}_k \subseteq \mathcal{N}$, which is selected randomly, without replacement, and check if the candidate point is  "good enough" for the function 
 \begin{equation*}\label{fdk} 
        F_{\mathcal{D}_k}(x_k,\mu_k):=\dfrac{1}{D_k}\sum\limits_{i\in\mathcal{D}_k} f_i(x_k)+\dfrac{\mu_k}{2}|| h(x_k)||^2,
    \end{equation*}
where $D_k=|\mathcal{D}_k|$. 
More precisely, we check if 
\begin{align}\label{relation}
F_{\mathcal{D}_k} (\overline{x}_k,\mu_k) \leq F_{\mathcal{D}_k} (x_k, \mu_k) - c \|\nabla F_{\D_k}(x_k,\mu_k)\|^2 + C \epsilon_k,
\end{align}
where $c$ and  $C$ are positive constants, arbitrarily small and arbitrarily large, respectively. 
This is an important step because it allows for evaluating the algorithm's progress on independent data, with low computational cost (in the experiments we use $D_k=1$, but any other choice such that $D_k \leq N_k$ is eligible). If the condition \eqref{relation} is satisfied, we accept the candidate point; otherwise, we set $x_{k+1}=x_k$. 

The additional check also serves as a test of the similarity of local cost functions (see e.g. \cite{krejic4} for more details). If the condition \eqref{relation} is met, then we decide that the current approximation is sufficient to describe the true objective function and keep the same sample size for the subsequent iteration. On the contrary, if the condition \eqref{line9} does not hold, we try to increase the level of precision and obtain a better approximation of the objective function by increasing the sample size to $N_{k+1}\in\{N_k+1,\dots, N\}$. This increase is arbitrary, which makes  ASPEN very flexible and adaptive to various types of problems. 

The penalty parameter in the MB phase is updated according to the constraint violation measure $\|h(x_k)\|$. If the current iterate is relatively far away from the feasible set, i.e., if $\|h(x_k)\|>\epsilon_k$, we increase the penalty parameter in order to encourage feasibility improvement. Otherwise, we keep it at the same level. 

The algorithm is stated as follows. 
%%%%%%%%%%%%%%%%%%%%%%%%%%%%%%%%%%%%%%%%%%%%%%
\begin{algorithm}
\caption{\textbf{ASPEN:} \\ \textbf{A}dditional \textbf{S}ampling \textbf{P}enalty method for \textbf{E}quality \textbf{N}onlinear constraints}
\begin{algorithmic}[1]
\State  Input: $x_0\in\mathbb{R}^n$, $N_0\in \mathbb{N}$, $c, \beta \in(0,1)$, $C,$ $\gamma$, $\mu_0>0$, $\{\epsilon_k\}_{k=0}^{\infty}.$ 
\State For  $k=0,1,2,...$  
    %\State
    \If{ $N_k<N$} 
    \State choose $\mathcal{N}_k\subseteq \mathcal{N}$
    \Else \State set $\mathcal{N}_k= \mathcal{N}$. 
    \EndIf
    \State  Calculate $p_k=-g_k$ via \eqref{gnk}.
    \State Find the smallest $j \in N_0$ such that $\alpha_k = \beta^j$ satisfies \eqref{line9}.
    \State  Set $\overline{x}_k = x_k + \alpha_k p_k$.
    \If{$N_k = N$}
        \State Set $x_{k+1} = \overline{x}_k$
        %\State
        \If {$\|\nabla F_{\N_k}(x_k, \mu_k)\|<  \frac{1}{\mu_k}$}\label{step13}
        \State Set $\mu_{k+1} = \gamma\mu_k$.
    \Else
        \State Set $\mu_{k+1} = \mu_k$.
    \EndIf

     \Else    
    
        \State Choose $\mathcal{D}_k\subseteq \mathcal{N}$ randomly and uniformly, without replacement.
        \If {\eqref{relation} holds }
        \State Set $x_{k+1} = \overline{x}_k$ and $N_{k+1} = N_k$.
        \Else
        \State Set $x_{k+1} = x_k$, and choose $N_{k+1} \in \{N_k + 1, \dots, N\}$.\label{line23}
        \EndIf
        \If {$\|h(x_k)\|> \epsilon_k$}\label{step25}
        \State Set $\mu_{k+1} = \gamma\mu_k$.\label{step26}
        \Else\label{step27}
        \State Set $\mu_{k+1} = \mu_k$.\label{step28}
        \EndIf\label{step29}
    
    \EndIf
\end{algorithmic}
\label{Alg1}
\end{algorithm}
%%%%%%%%%%%%%%%%%%%%%%%%%%%%%%%%%%%%%%%%%%%%%%
\newpage
\section{Convergence analysis}\label{sec4}
Within this section, we prove almost sure convergence results for ASPEN. The analysis is conducted by observing two possible scenarios (MB and FS) separately, and integrating them at the end. We start by stating the following assumption.

\begin{assumption}\label{assumption1}
The functions $f_i$, $i=1,\ldots, N$ are bounded from below and continuously differentiable with  $L$-Lipschitz continuous gradients. Moreover, the constraints violation function $q$ is continuously differentiable with  $L$-Lipschitz continuous gradient. 
\end{assumption}
Notice that this assumption implies that all the sampled functions $f_{\N_k}$ are also bounded from below and continuously differentiable with  $L$-Lipschitz continuous gradients. Notice that the penalty functions  $F_{\N_k}$ are also bounded from below and continuously differentiable. Moreover, the gradients $\nabla F_{\N_k}$ are also Lipschitz-continuous with the  Lipschitz constant $(1+\mu_k)L.$ Since the line search is performed with the negative gradient direction of the approximate penalty function, it can be shown that 
\begin{equation} \label{baralfa}
\alpha_k \geq \min \left\lbrace1, \frac{2 \beta (1-\eta) }{(1+\mu_k)L}\right\rbrace=:\bar{\alpha}_{\mu_k}.
\end{equation}
Now, let us consider the MB phase of ASPEN. As usual for an additional sampling-based method, we define the following sets needed for the analysis. By  $\mathcal{D}_k^+$ we denote the set of all possible outcomes of $\mathcal{D}_k$ at iteration $k$ that satisfy \eqref{relation}, i.e.,  
\begin{align*}
\mathcal{D}_k^+=\{\mathcal{D}_k\subseteq\mathcal{N} \mid F_{\mathcal{D}_k} (\overline{x}_k,\mu_k) \leq F_{\mathcal{D}_k} (x_k, \mu_k) - c \|\nabla F_{\D_k}(x_k,\mu_k)\|^2 + C \epsilon_k\},
\end{align*}
and by  $\mathcal{D}_k^-$ be the complementary subset, i.e.,  
\begin{align*}
\mathcal{D}_k^-=\{\mathcal{D}_k\subseteq\mathcal{N} \mid F_{\mathcal{D}_k} (\overline{x}_k,\mu_k) > F_{\mathcal{D}_k} (x_k, \mu_k) - c \|\nabla F_{\D_k}(x_k,\mu_k)\|^2 + C \epsilon_k\}.
\end{align*}

The following lemma shows that, if  ASPEN stays in the MB phase during the whole optimization process, then the condition \eqref{relation} is satisfied for all possible choices of $\D_k$. The proof of this lemma is essentially the same as the proof of Lemma 1 \cite{krejic4}, but we state it here for completeness. 

\begin{lemma}
\label{lemma1}
Suppose that Assumption A\ref{assumption1} holds. If $N_k < N$ for all $k \in \mathbb{N}$, then a.s. there exists $k_1 \in\mathbb{N}$ such that $\mathcal{D}_k^- =\emptyset$ for all $k \geq k_1$.
\end{lemma}
\begin{proof}
Let us assume, aiming for contradiction, that the lemma does not hold. Then there exists an infinite set of indices $K\subseteq\mathbb{N}$ such that $\D_k^{-} = \emptyset $ for all $k\in K$.
Since the sample size sequence is non-decreasing and satisfies $N_k<N$ for all $k\in\mathbb{N}$ by the assumption of this lemma, there must exist some $k_1\in\mathbb{N}$ such that $N_k=\bar{N}<N$ for all $k\geq k_1$. This implies that the number of elements in the subset $\D_k$ satisfies $
D_k\leq N_k\leq\bar{N}\leq N-1.$
Since  $\D_k$ is selected randomly and uniformly from a finite collection, there exists a constant $q>0$ such that
$$
\mathbb{P}(\D_k \in D_k^{-})\geq q\quad\text{for all } k\in K.
$$
Moreover, without loss of generality, we may assume that $K\subseteq\{k \in \mathbb{N}:k\geq k_1\}$, i.e., every $k\in K$ satisfies $k\geq k_1$. Therefore, choosing $\D_k \in D_k^{+} $ for all $k \in K$ is a.s. impossible since 
$$
\mathbb{P}(\D_k \in  \D_k^{+},  k\in K)\leq\prod_{k\in K}(1-q)=0.
$$
This means that a.s. there exists an iteration $\tilde{k} \in K $ such that
$$
\D_{\tilde{k}}\in \D_{\tilde{k}}^{-}.
$$
According to the algorithm, this implies that 
$ N_{\tilde{k}+1}>N_{\tilde{k}}=\bar{N},
$
contradicting the assumption that $ N_k = \bar{N}$ for all $k\geq k_1$. 
Hence, the statement holds.
\end{proof}

Again, following the concept of additional sampling, one can show that the Armijo-like condition related to the full sample penalty function is satisfied for all $k$ large enough, even in the MB scenario.  

\begin{lemma}\label{lemma2}
Suppose that the assumptions of Lemma \ref{lemma1} hold. Then, a.s. the following holds for all $k\geq k_1$ 
$$  F(x_{k+1}, \mu_k)\leq F(x_k, \mu_k)-c \|\nabla F(x_k, \mu_k)\|^2+C \epsilon_k.
$$
\end{lemma}
\begin{proof}   
First, recall that Lemma \ref{lemma1} implies that a.s. for all $k \geq k_1$ we have $D^-_k = \emptyset$ and thus the condition \eqref{relation} is satisfied.  Therefore, the candidate point is accepted for all $k$ large enough, i.e.,   $x_{k+1}=\bar{x}_k$ for all  $k \geq k_1$. Moreover, notice that $D^-_k = \emptyset$ also implies\footnote{Since $\D_k$ is chosen uniformly, with replacement, one possible choice for $\D_k$ is $\{i,\ldots, i\}$. If there exists $i$ that violates the condition \eqref{fi}, then the condition \eqref{relation} would be violated for $\D_k=\{i,...,i\}$ which would imply that $D^-_k \neq  \emptyset$.} that for each $i \in \N$ we have 
\begin{equation} \label{fi}
   F_i(x_{k+1}, \mu_k)\leq F_i(x_k, \mu_k)-c \|\nabla F_i(x_k, \mu_k)\|^2+C \epsilon_k,
\end{equation} 
where 
$$F_i(x, \mu_k):=f_i(x_k)+\dfrac{\mu_k}{2}|| h(x)||^2.$$
Thus, considering \eqref{fi}, by summing up and dividing by $N$ we obtain that 
\begin{equation} \label{f}
   F(x_{k+1}, \mu_k)\leq F(x_k, \mu_k)-c\sum_{i=1}^{N} \|\nabla F_i(x_k, \mu_k)\|^2+C \epsilon_k
\end{equation} 
for all $k \geq k_1$. 
Finally, using the convexity of the norm, we obtain 
$$\|\nabla F(x_k, \mu_k)\|^2=\|\frac{1}{N} \sum_{i=1}^{N} \nabla F_i(x_k, \mu_k)\|^2 \leq \frac{1}{N}\sum_{i=1}^{N} \|\nabla F_i(x_k, \mu_k)\|^2.$$
Combining this with \eqref{f}, we conclude the proof. 
\end{proof}
 Before stating the next theorem, let us denote by $\mathbb{E}_{MB}(\cdot)$ and $\mathbb{E}_{FS}(\cdot)$ the conditional expectation concerning all the sample paths of ASPEN that fall into the MB and FS scenario, respectively. Similarly, we define the corresponding conditional  probabilities $\mathbb{P}_{MB}(\cdot)$ and $\mathbb{P}_{FS}(\cdot)$. The proof is essentially the same as the proof of Theorem 1 in \cite{lsnmbb} e.g., but we state its short version for completeness. 

\begin{theorem}\label{theorem1}
Suppose that Assumption A\ref{assumption1} holds and that the sequence $\{x_k\}_{k \in \mathbb{N}}$ generated by ASPEN is bounded. Then 
$$\mathbb{P}_{MB}(\lim_{k \to \infty} \nabla F(x_k, \mu_k)=0)=1.$$
\end{theorem}
\begin{proof} Since we observe only the  MB scenario sample paths, Lemma  \ref{lemma2} implies that a.s. the following holds for each $l \in \mathbb{N}$
$$  F(x_{k_1+l}, \mu_{k_1+l})\leq F(x_{k_1}, \mu_{k_1})-c \sum_{j=0}^{l-1} \|\nabla F(x_{k_1+j}, \mu_{k_1+j})\|^2+C \sum_{j=0}^{l-1} \epsilon_{k_1+j}.
$$
Due to summability of $\epsilon_k$ and  boundedness of iterates, applying the conditional expectation $\mathbb{E}_{MB}(\cdot)$ and letting $l \to \infty$ we obtain 
$$\sum_{j=0}^{\infty} \mathbb{E}_{MB}(\|\nabla F(x_{k_1+j}, \mu_{k_1+j})\|^2) < \infty$$
and the result follows from the extended form of Markov's inequality and the Borel-Cantelli lemma. 
\end{proof}
Next, we prove a.s. convergence towards a KKT point in the MB case by considering different scenarios with respect to the penalty parameter sequence. The proof is conducted under the Linear Independence Constraint Qualification (LICQ) assumption. The second part of the proof, considering unbounded $\mu_k$ follows the same steps as the deterministic quadratic penalty method analysis, but we state it here for completeness. 

\begin{theorem}\label{theorem2}
Let assumptions of Theorem \ref{theorem1} hold and assume that $N_k <N$ for every $k \in \mathbb{N}$. Then, a.s., every accumulation point of the sequence $\{x_k\}_{k \in \mathbb{N}}$ at which LICQ holds is a KKT point of problem \eqref{OriginalProblem}. 
\end{theorem}
\begin{proof} Recall that  Theorem \ref{theorem1} implies that $\lim_{k \to \infty} \nabla F(x_k, \mu_k)=0$ a.s. Let $x^*$ be an arbitrary accumulation point of the considered sequence,  i.e., let 
$$\lim_{k \in K} x_k=x^*.$$ Under the MB scenario,  we distinguish two possible cases regarding the penalty parameter - bounded and unbounded  $\mu_k$. 

First, let us assume that $\mu_k$ is bounded. Since the sequence of penalty parameters is non-decreasing, this further implies the existence of  $\bar{\mu}$ such that $\mu_k=\bar{\mu}$ for all $k$ large enough. According to lines \ref{step25}-\ref{step29} of ASPEN,  for all $k$ large enough we have $\mu_{k+1}=\mu_k$ and $\|h(x_k)\|\leq \epsilon_k$. Since $\{\epsilon_k\}_{k \in \mathbb{N}}$ is assumed to be summable, we know that $\lim_{k \to \infty} \epsilon_k=0$, which further implies 
$$\lim_{k \to \infty} h(x_k)=0.$$
Thus, each accumulation point of $\{x_k\}_{k \in \mathbb{N}}$ is feasible and we have $h(x^*)=0$ and therefore   
$$0=\lim_{k \in K} \nabla F(x_k, \mu_k)=\lim_{k \in K} (\nabla f (x_k)+ \mu_k \nabla^T h(x_k) h(x_k))=\nabla f (x^*),$$
holds a.s. and the statement is proved.  

Now, let us consider the case where $\lim_{k \to \infty} \mu_k=\infty$. Then, according to \eqref{gnk} we have 
$$\|\nabla^T h(x_k) h(x_k)\|\leq \frac{1}{\mu_k} (\|\nabla F(x_k, \mu_k)\|+\|\nabla f (x_k)\|)$$
and taking the limit over $K$ we obtain 
$$\|\nabla^T h(x^*) h(x^*)\|=0$$
and the LICQ condition implies $h(x^*)=0$.  Moreover, LICQ also implies that $\nabla h(x^*)\nabla^T h(x^*)$ is non-singular and, due to continuity,  $\nabla h(x_k)\nabla^T h(x_k)$ is also non-singular for each $k \in K$ large enough. Thus, considering \eqref{gnk} again and defining $\lambda_k:=\mu_k h(x_k)$ we obtain 
 \begin{equation}\label{gnk2} 
        \lambda_k =(\nabla h(x_k)\nabla^T h(x_k))^{-1} \nabla h(x_k) ( \nabla F(x_k, \mu_k) -\nabla f(x_k))
    \end{equation}
    for all $k$ large enough and a.s. 
 \begin{equation}\label{gnk3} 
        \lim_{k \in K} \lambda_k =-(\nabla h(x^*)\nabla^T h(x^*))^{-1} \nabla h(x^*) \nabla f(x^*)=:\lambda^*. 
    \end{equation}
    Thus, a.s., 
 \begin{equation}\label{gnk4} 
        0=\lim_{k \in K} \nabla F(x_k, \mu_k)=\nabla f(x^*)+\nabla^T h(x^*)\lambda^*,
    \end{equation}
which together with $h(x^*)=0$ implies that $x^*$ is a KKT point.
\end{proof}

Next, we prove a.s. convergence for the FS scenario.
\begin{theorem}\label{theorem3}
Suppose that Assumption A\ref{assumption1} holds and that the sequence $\{x_k\}_{k \in \mathbb{N}}$ generated by ASPEN is bounded. Moreover, suppose that there exists $\tilde{k} $ such that $N_k=N$ for all $k \geq \tilde{k}$. Then, a.s., there exists an accumulation point of the sequence $\{x_k\}_{k \in \mathbb{N}}$ which is a KKT point of problem \eqref{OriginalProblem} provided that  LICQ holds at that point. 
\end{theorem}

\begin{proof}
Let us consider iterations $k \geq \tilde{k}$. Then, $\N_k=\N$ and the line search implies 
$$  F(x_{k+1}, \mu_k)\leq F(x_k, \mu_k)-\eta \alpha_k\|\nabla F(x_k, \mu_k)\|^2+ \epsilon_k.
$$
Notice that in the FS phase, the penalty parameter $\mu_k$ is increased only if $\|\nabla F(x_k, \mu_k)\|<  1/\mu_k$. On the other hand, if $\mu_k$ is kept fixed on some $\bar{\mu}$, from \eqref{baralfa} and the line search we obtain 
$$  F(x_{k+1}, \bar{\mu})\leq F(x_k, \bar{\mu})-\eta \bar{\alpha}_{\bar{\mu}}\|\nabla F(x_k, \bar{\mu})\|^2+ \epsilon_k
$$
and, due to boundedness of iterates and summability of $\epsilon_k$, we conclude that $\|\nabla F(x_k, \bar{\mu})\|^2$ tends to zero. This further implies that after a finite number of iterations we will have $\|\nabla F(x_k, \bar{\mu})\|<1/\bar{\mu}$. Therefore, we conclude that the penalty parameter cannot be fixed for an infinite number of iterations. In fact, it must be increased infinitely many times, and thus $\lim_{k\to\infty} \mu_k=\infty$ holds. Moreover, we conclude that there exists a subset of iterations $\tilde{K}$ such that 
$\|\nabla F(x_k, \mu_k)\|<  1/\mu_k$ for all $k \in \tilde{K}$ which further implies that 
$$\lim_{k \in \tilde{K}} \nabla F(x_k, \mu_k)=0.$$ Since the sequence of iterates is bounded, there exist $x^*$ and  $\tilde{K}_1 \subseteq \tilde{K}$ such that $\lim_{k \in \tilde{K}_1} x_k=x^*$ and following the steps of the second part of the proof of Theorem \ref{theorem2} we obtain the result. 
\end{proof}

Finally, considering both possible scenarios (FS and MB), we obtain the main result for the ASPEN method. 

\begin{theorem}
    Suppose that Assumption A\ref{assumption1} holds and that the sequence $\{x_k\}_{k \in \mathbb{N}}$ generated by ASPEN is bounded.  Then, a.s., there exists an accumulation point of the sequence $\{x_k\}_{k \in \mathbb{N}}$ which is a KKT point of problem \eqref{OriginalProblem} provided that  LICQ holds.
\end{theorem} 

\section{Numerical results}\label{sec5}

In this section, we evaluate the performance of the proposed method on real-data machine learning tasks, considering several binary classification datasets from  LIBSVM collection \cite{libsvm} listed in  Table \ref{tab:datasets}. Moreover, to provide further insights into the behavior of the proposed method, we also consider an academic problem (HS24) from the CUTEst collection~\cite{cutest} and modify it by adding different levels of noise. Other benchmark datasets frequently referenced in the literature include those from the UCI Machine Learning Repository \cite{dataset1} and the MNIST handwritten digit database \cite{dataset2}.

We begin our numerical study by demonstrating that ASPEN offers benefits compared to the deterministic approach, where $N_k=N$ for all $k \in \mathbb{N}$.  We call this method  "Full" since it uses the full sample during the whole optimization process. More precisely, Full follows the standard penalty approach by solving each subproblem approximately - until $\|\nabla F_{\N_k}(x_k, \mu_k)\|<  1/\mu_k$ is satisfied, and then it increases the penalty parameter by $\mu_{k+1}=\gamma \mu_k$. Furthermore, we also compare ASPEN to a heuristic method named "Heur"  which starts with a subsample of size $N_0$ and increases it by $N_{k+1} = \min \lbrace \lceil 1.1 N_k \rceil, N \rbrace)$  whenever $\|\nabla F_{\N_k}(x_k, \mu_k)\|<  1/\mu_k$  happens. The penalty parameter is updated as in the Full method, and the same non-monotone line search is applied for solving the subproblems for all the methods mentioned above. 

The parameters of ASPEN are the following. The starting penalty parameter is $\mu_0=1$, while the initial sample size is $N_0=\lceil 0.01N\rceil$. The additional sampling is applied with $ D_k=1,$ $ c=10^{-4}$  and $  C=1$, while the line  search is performed with $ \beta=0.1, $ $\eta=10^{-4}$  and $\epsilon_k=k^{-1.1}$. When needed, the sample size of ASPEN is increased by one, i.e., $N_{k+1}=N_k+1$ in line \ref{line23} of ASPEN. Heur uses the same starting values for the penalty parameter and the sample size. We set $\gamma=1.1$ for all the considered methods. Starting points $x_0$ are equal for all the considered methods and are obtained by normalizing random vectors with a Gaussian distribution.

\begin{table}[htb!]
\centering
\caption{Binary classification data set details~\cite{libsvm}.}  
\begin{tabular}{|l|c|c|}
\hline
\textbf{Dataset} & \textbf{Dimension ($n$)} & \textbf{Datapoints ($N$)} \\
\hline
a9a         & 123 & 32,561 \\
Australian  & 14  & 690    \\
Heart       & 13  & 270    \\
Mushrooms   & 112 & 8,124  \\
Splice      & 60  & 1,000  \\
MNIST & 784 & 60000 \\
\hline
\end{tabular}
\label{tab:datasets}
\end{table}

We consider constrained logistic regression binary classification problems as in \cite{Numerika}, commonly used in machine learning applications. The form is the following 
\begin{align}
\label{problem51}
  \min_{x \in \mathbb{R}^n} f(x) = \frac{1}{N} \sum_{i=1}^{N} \log \left(1 + e^{-b_i a_i^\top x} \right) \quad \text{s.t.}   \quad \|x\|_2^2 = 1,
\end{align}
where $a_i \in \mathbb{R}^n$ and $b_i \in \{-1,1\}$ represent the attributes and the corresponding label for the  $i$-th data point, respectively.  We model the computational cost by $FEV_k$ - the number of scalar products required by the specified method to compute $x_k$, starting from the initial point $x_0$.

To evaluate the performance of the considered methods,  we show: \\
1) the distance between $x_k$ and the solution $x^*$ of the considered problem, i.e., $||x_k-x^*||$,  against computational cost measure  $FEV_k$ in graphs a); \\
2) the sample size behavior across iterations, graphs b); \\
3)  the penalty parameter update across iterations, graphs c).   \\
The results for all the datasets from Table \ref{tab:datasets} in Figures \ref{MNIST_gama11}-\ref{australian_gama11} are presented.

The results show that ASPEN manages to outperform Full and Heur on most of the datasets, especially in light of achieving a better vicinity of the solution with significantly lower computational costs. The results also confirm the adaptive nature of ASPEN, especially when the sample size is considered. The graphs b) reveal that the sample size is increased according to the problem at hand, therefore highlighting ASPEN's data-driven adaptivity. For instance, the MNIST dataset (Fig. \ref{MNIST_gama11}) requires faster sample size growth to cope with the diversity of the data, while the Mushrooms dataset (Fig. \ref{Mushrooms_gama11}) obviously contains more similar data points, which allows good approximate solutions even under modest sample sizes which practically fall into a mini-batch framework. Interestingly, the full sample is not reached in any of the considered problems. Furthermore, as expected, the penalty parameter is increased more rapidly for ASPEN than for the other two methods (graphs c), but according to the optimality gap, it seems to be beneficial -  it allows the algorithm to progressively enforce feasibility while maintaining efficiency. 

\begin{figure}[htb!]
    \centering
    \includegraphics[width=\textwidth]{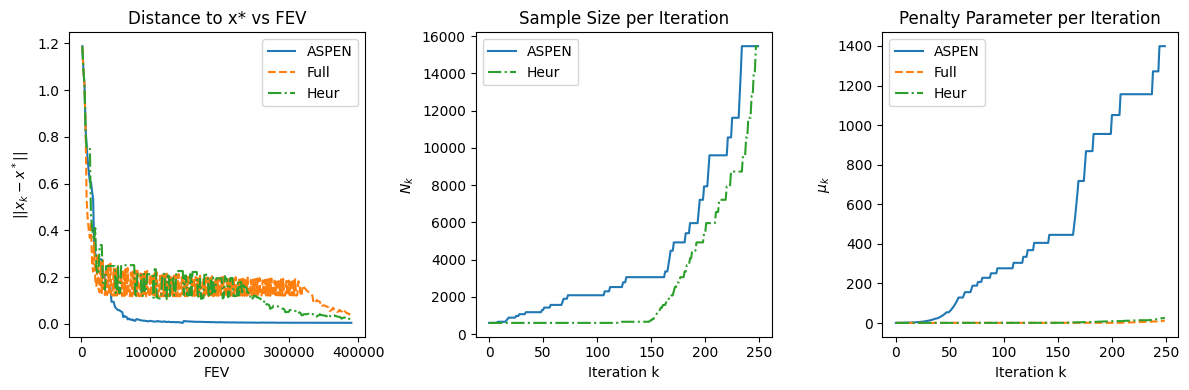}
     \makebox[\textwidth]{%
        \begin{minipage}[t]{0.33\textwidth}
            \centering
            \footnotesize a)
        \end{minipage}%
        \begin{minipage}[t]{0.33\textwidth}
            \centering
            \footnotesize b)
        \end{minipage}%
        \begin{minipage}[t]{0.33\textwidth}
            \centering
            \footnotesize c)
        \end{minipage}%
    }
    \caption{\footnotesize{ \textit{MNIST} dataset. ASPEN vs. Full and Heur:  optimality gap vs. FEV (a); sample size vs. iteration (b);  penalty parameter vs. iteration (c). }}
     \label{MNIST_gama11}
\end{figure}

\begin{figure}[htb!]
    \centering
    \includegraphics[width=\textwidth]{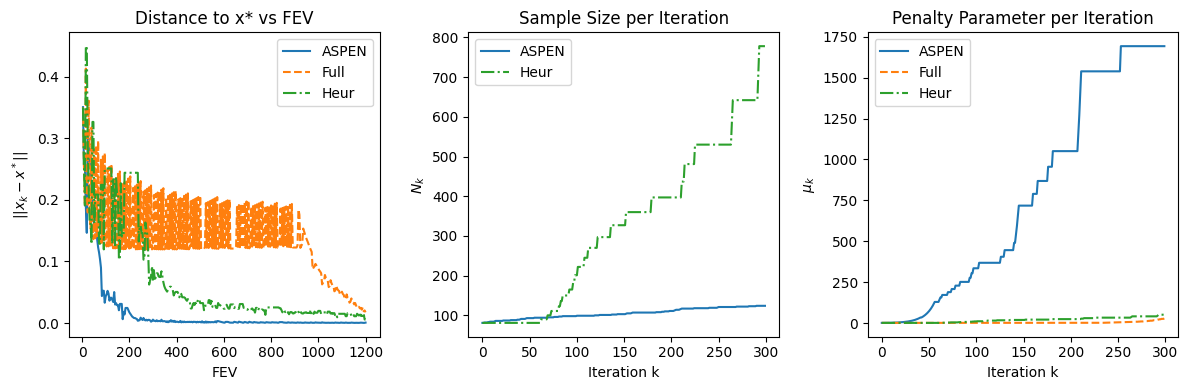}
     \makebox[\textwidth]{%
        \begin{minipage}[t]{0.33\textwidth}
            \centering
            \footnotesize a)
        \end{minipage}%
        \begin{minipage}[t]{0.33\textwidth}
            \centering
            \footnotesize b)
        \end{minipage}%
        \begin{minipage}[t]{0.33\textwidth}
            \centering
            \footnotesize c)
        \end{minipage}%
    }
    \caption{\footnotesize{ \textit{Mushrooms} dataset. ASPEN vs. Full and Heur:  optimality gap vs. FEV (a); sample size vs. iteration (b);  penalty parameter vs. iteration (c).  }}
     \label{Mushrooms_gama11}
\end{figure}

\begin{figure}[htb!]
    \centering
    \includegraphics[width=\textwidth]{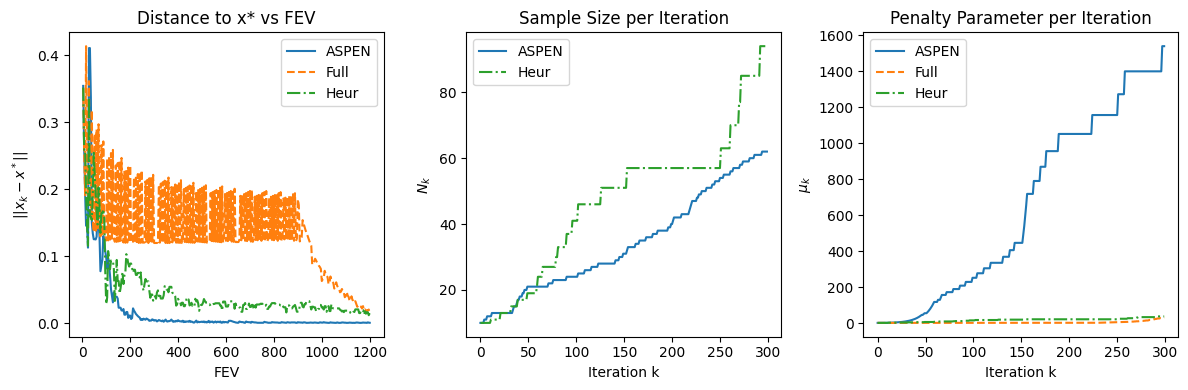}
     \makebox[\textwidth]{%
        \begin{minipage}[t]{0.33\textwidth}
            \centering
            \footnotesize a)
        \end{minipage}%
        \begin{minipage}[t]{0.33\textwidth}
            \centering
            \footnotesize b)
        \end{minipage}%
        \begin{minipage}[t]{0.33\textwidth}
            \centering
            \footnotesize c)
        \end{minipage}%
    }
    \caption{\footnotesize{ \textit{Splice} dataset. ASPEN vs. Full and Heur:  optimality gap vs. FEV (a); sample size vs. iteration (b);  penalty parameter vs. iteration (c).}}
     \label{Splice_gama11}
\end{figure}

\begin{figure}[htb!]
    \centering
    \includegraphics[width=\textwidth]{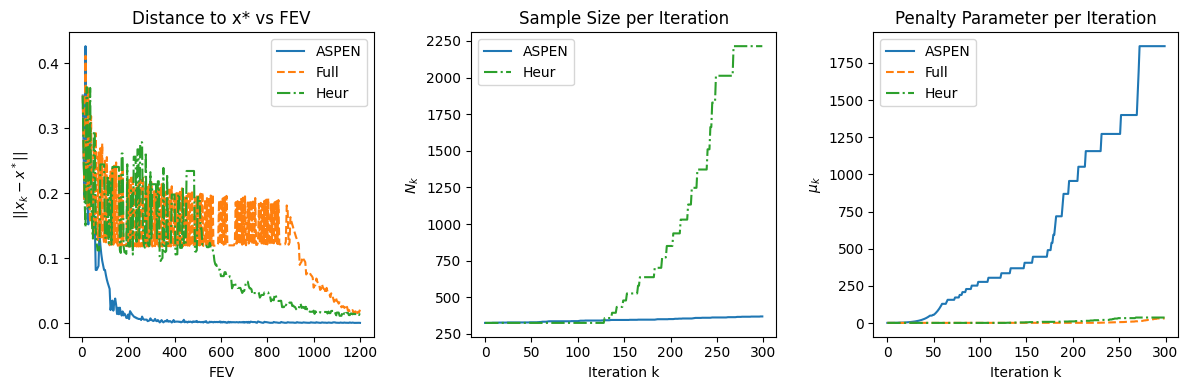}
     \makebox[\textwidth]{%
        \begin{minipage}[t]{0.33\textwidth}
            \centering
            \footnotesize a)
        \end{minipage}%
        \begin{minipage}[t]{0.33\textwidth}
            \centering
            \footnotesize b)
        \end{minipage}%
        \begin{minipage}[t]{0.33\textwidth}
            \centering
            \footnotesize c)
        \end{minipage}%
    }
    \caption{\footnotesize{ \textit{a9a} dataset. ASPEN vs. Full and Heur:  optimality gap vs. FEV (a); sample size vs. iteration (b);  penalty parameter vs. iteration (c).}}
     \label{a9a_gama11}
\end{figure}

\begin{figure}[H]
    \centering
    \includegraphics[width=\textwidth]{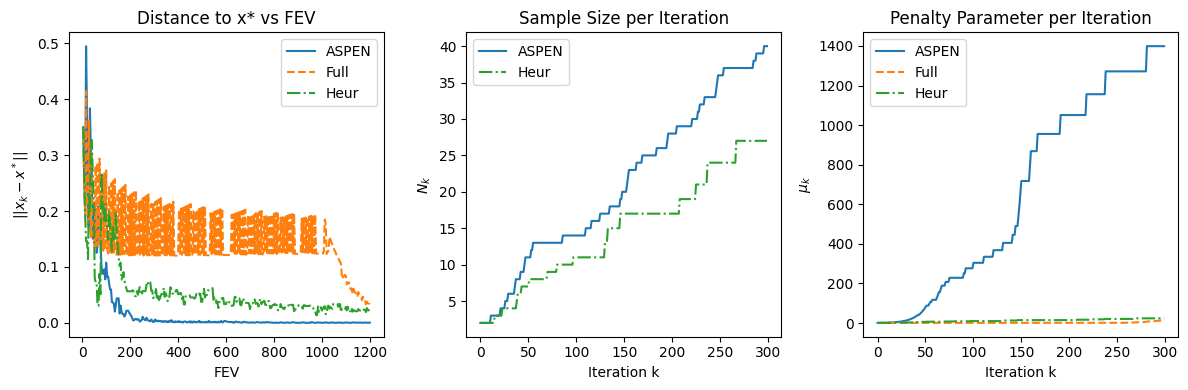}
     \makebox[\textwidth]{%
        \begin{minipage}[t]{0.33\textwidth}
            \centering
            \footnotesize a)
        \end{minipage}%
        \begin{minipage}[t]{0.33\textwidth}
            \centering
            \footnotesize b)
        \end{minipage}%
        \begin{minipage}[t]{0.33\textwidth}
            \centering
            \footnotesize c)
        \end{minipage}%
    }
    \caption{\footnotesize{ \textit{Heart} dataset. ASPEN vs. Full and Heur:  optimality gap vs. FEV (a); sample size vs. iteration (b);  penalty parameter vs. iteration (c).}  }
     \label{heart_gama11}
\end{figure}

\begin{figure}[H]
    \centering
    \includegraphics[width=\textwidth]{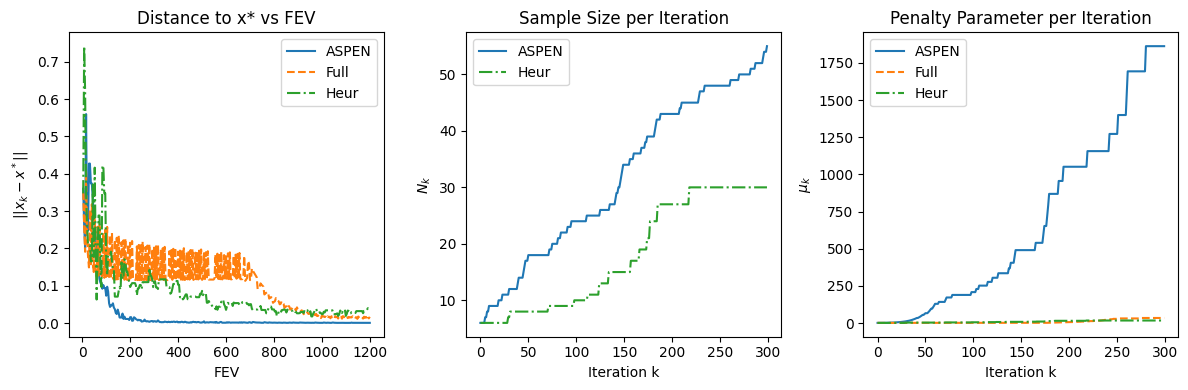}
     \makebox[\textwidth]{%
        \begin{minipage}[t]{0.33\textwidth}
            \centering
            \footnotesize a)
        \end{minipage}%
        \begin{minipage}[t]{0.33\textwidth}
            \centering
            \footnotesize b)
        \end{minipage}%
        \begin{minipage}[t]{0.33\textwidth}
            \centering
            \footnotesize c)
        \end{minipage}%
    }
    \caption{\footnotesize{ \textit{Australian} dataset. ASPEN vs. Full and Heur:  optimality gap vs. FEV (a); sample size vs. iteration (b);  penalty parameter vs. iteration (c).}  }
     \label{australian_gama11}
\end{figure}

In the sequel, we compare ASPEN to the state-of-the-art stochastic optimization methods of the relevant framework: Sto-SQP\cite{berahas2} and SVR-STO\cite{Numerika}. While the ASPEN keeps the same parameter settings as in the previous experiments, the selected parameters for the competing methods are based on empirical tuning and the recommendations from the literature. After validation, the values were fixed and consistently applied across all experiments, as was done in  \cite{Numerika}. In detail, the setup is the following: Sto-SQP  uses  $\theta = 10^{4},\; \tilde{\tau}_{-1}=0.1,\; \epsilon_{\tau}=10^{-6},\;
\tilde{\xi}_{-1}=0.1,\; \epsilon_{\xi}=10^{-2},\; \sigma_{Sto-SQP} = 0.5$; 
    SVR-STO is employed with  $\sigma_{SVR-STO} = 0.5,\; \theta = 10^{4},\; \tilde{\tau}_{-1,0} = 0.1,\; \epsilon_{\tau}=10^{-6},\;\alpha_u = 10^{6},\; \beta_{SVR-STO} = 1$.
Considering Figure \ref{comparison1}, ASPEN shows to be competitive with the state-of-the-art methods, outperforming them on most of the considered datasets. The detailed analysis is provided below. 

\begin{figure}[ht!]
    \centering
    \begin{minipage}[c]{0.48\textwidth}
        \centering
        \includegraphics[width=\textwidth]{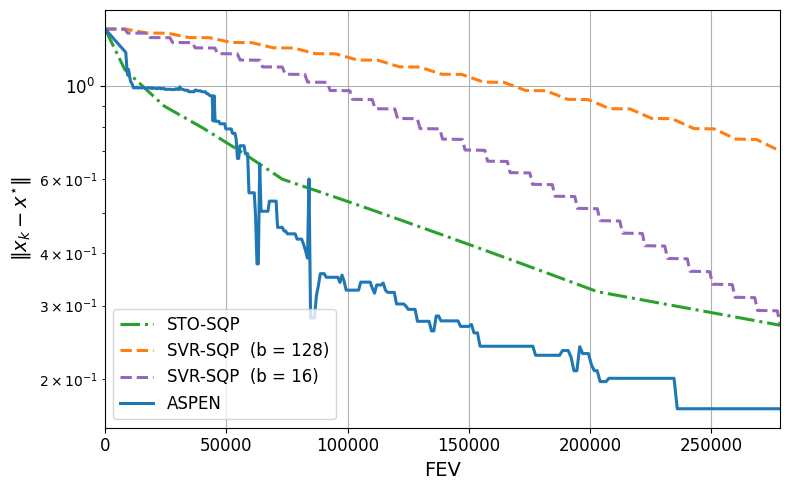}\\
        \footnotesize a)
    \end{minipage}
    \hfill
    \begin{minipage}[c]{0.48\textwidth}
        \centering
        \includegraphics[width=\textwidth]{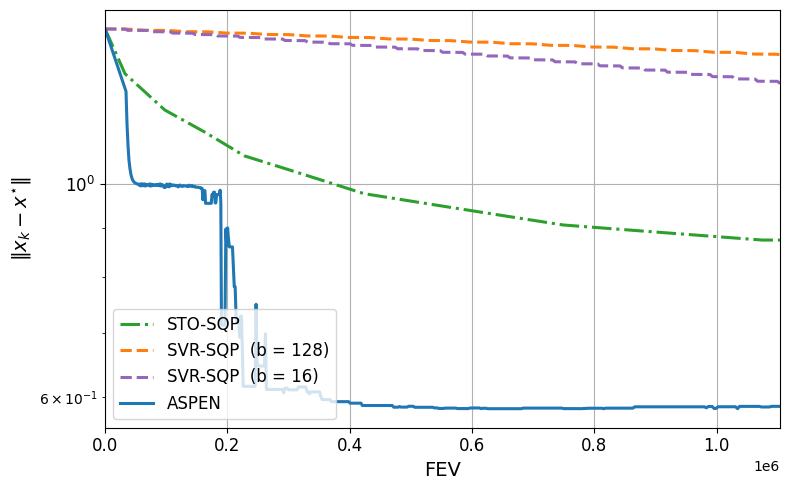}\\
        \footnotesize b)
    \end{minipage}
    \vspace{1em}
        \begin{minipage}[c]{0.48\textwidth}
        \centering
        \includegraphics[width=\textwidth]{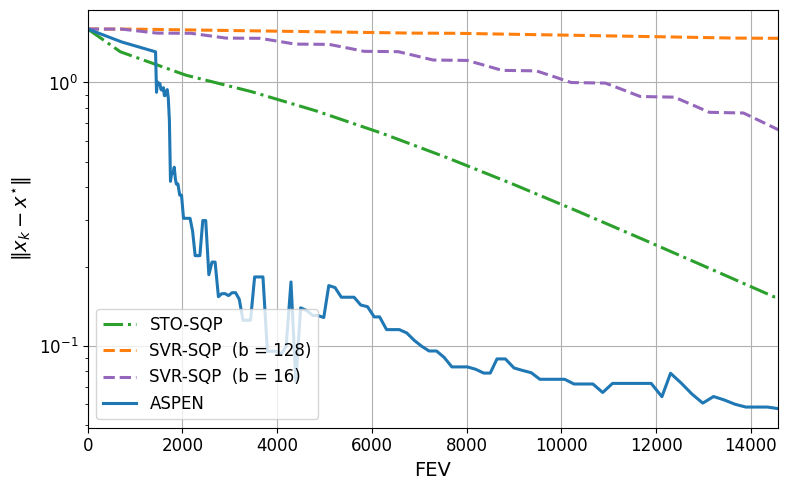}\\
        \footnotesize c)
    \end{minipage}
    \hfill
    \begin{minipage}[c]{0.48\textwidth}
        \centering
        \includegraphics[width=\textwidth]{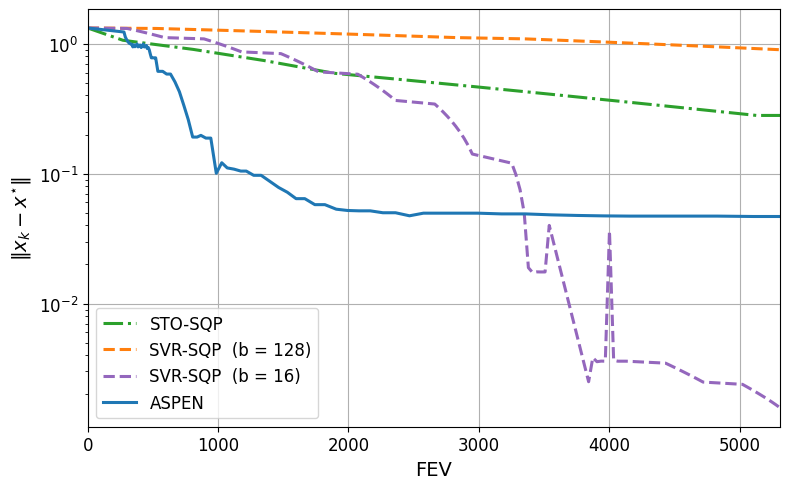}\\
        \footnotesize d)
    \end{minipage}
    \vspace{1em}
        \begin{minipage}[c]{0.48\textwidth}
        \centering
        \includegraphics[width=\textwidth]{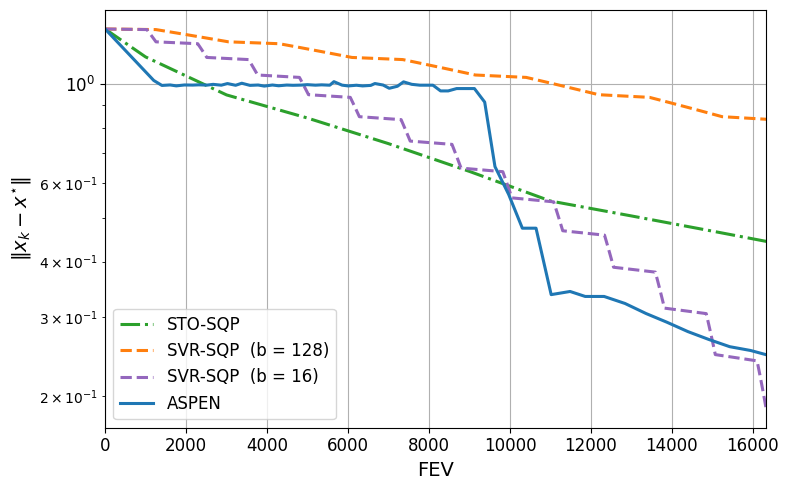}\\
        \footnotesize e)
    \end{minipage}
    \begin{minipage}[c]{0.48\textwidth}
        \centering
        \includegraphics[width=\textwidth]{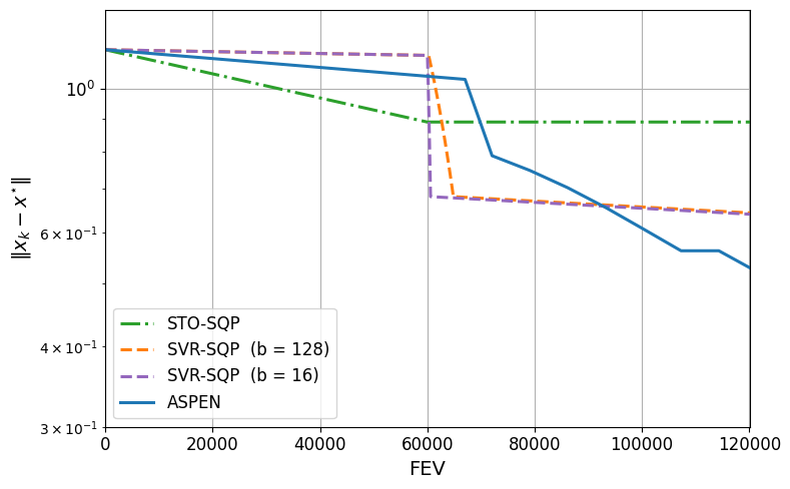}\\
        \footnotesize f)
    \end{minipage}
    \hfill

    \caption{\footnotesize{ASPEN vs. \text{STO-SQP} and two variants of \text{SVR-SQP}. Distance to the optimal solution versus the number of scalar products (FEV):  a) \textit{Mushrooms dataset}; b) \textit{a9a dataset}; c) \textit{Australian dataset}; d) \textit{Heart dataset}; e) \textit{Splice dataset}; f) \textit{MNIST dataset}.  }} 
    \label{comparison1}
\end{figure}

Figure \ref{comparison1} a) – \textit{Mushrooms dataset}. \text{ASPEN's} optimality gap  drops down sharply after only a few FEVs and continues to decline throughout the interval. \text{STO-SQP} shows a stable, monotone decrease but with a much smaller slope, indicating a higher computational cost for the same accuracy. \text{SVR-SQP} with $b = 16$ converges more slowly than \text{STO-SQP}, whereas the \text{SVR-SQP}  with $b = 128$  yields only negligible improvement within the available FEV budget.

Figure \ref{comparison1} b) – \textit{a9a dataset}. \text{ASPEN} is the only method achieving a pronounced error reduction; after a brief plateau, it resumes decreasing, underscoring the effectiveness of its adaptive scheme. \text{STO-SQP} steadily lowers the distance, albeit at a moderate rate. Both \text{SVR-SQP} variants display limited convergence, suggesting that this dataset would require more iterations or a different batch-size schedule.

Figure \ref{comparison1} c) – \textit{Australian dataset}. \text{ASPEN} again attains the largest error reduction and maintains the lowest optimality gap across the entire FEV horizon. \text{STO-SQP} provides a uniform but slower decrease. \text{SVR-SQP} with $b = 16$ accelerates in the later phase and approaches \text{STO-SQP}’s accuracy, highlighting the stochastic variant’s sensitivity to the data’s statistical properties. The $b = 128$ variant shows the smallest improvement, confirming that large batch sizes can hamper stochastic acceleration under a fixed evaluation budget.

Figure \ref{comparison1} d) – \textit{Heart dataset}. ASPEN demonstrates the fastest initial convergence, rapidly reducing the distance to the optimal solution $||x_k-x^*||$ below $10^{-1}$ within just a few hundred FEVs. However, after this sharp drop, its progress stagnates, maintaining a nearly constant level. Interestingly, SVR-SQP $(b = 16)$ continues to steadily decrease over time and eventually surpasses \text{ASPEN} in terms of final accuracy, achieving the lowest distance to the optimum. \text{STO-SQP} and \text{SVR-SQP} $(b = 128)$ exhibit slower and more gradual convergence, remaining less competitive throughout. 

Figure \ref{comparison1} e) – \textit{ Splice dataset}. On this dataset,  \text{SVR-SQP} $(b = 16)$ eventually outperforms \text{ASPEN} in terms of final accuracy, although \text{ASPEN} demonstrates a significantly faster initial convergence. \text{ASPEN} again achieves strong early convergence but experiences prolonged stagnation between $2000$ and $9000$ FEV, only improving afterward. In contrast, \text{SVR-SQP} $(b = 16)$ exhibits consistent and monotonic progress and ultimately outperforms all other methods by achieving the best final solution. \text{STO-SQP} remains moderately effective, while \text{SVR-SQP} $(b = 128)$ shows the slowest rate of improvement. These results highlight that while ASPEN excels in fast early convergence, \text{SVR-SQP} $(b = 16)$ demonstrates superior long-term accuracy on both datasets.

Figure 7 f) – MNIST dataset. On this dataset, for the binary classification problem, we can conclude that the \text{ASPEN} algorithm is the most successful in terms of solution accuracy, although it requires slightly more function evaluations to outperform the other algorithms, which are also reliable in terms of convergence but somewhat slower. The \text{SVR-SQP} algorithms (with batch sizes $b=16$ and $b=128$) produced better results than \text{STO-SQP} in scenarios with a larger number of FEV, even though \text{STO-SQP} initially shows the best performance.

We end this section by providing some more insights on the sample size behavior of the proposed method. To this end,   we consider an academic problem (HS24) from the CUTEst collection~\cite{cutest}, modified by introducing a Gaussian noise. More precisely, in order to simulate a stochastic environment,  we consider a perturbed problem
\begin{equation}
\label{problem52}
   \min_{x \in \mathbb{R}^n} f(x):=\frac{1}{N}\sum_{i=1}^N (\tilde{f}(x)+\varepsilon_i^2||x||^2)  \quad \text{s.t.}   \quad  \|x\|_2^2 = 1,
\end{equation}
where 
$\tilde {f}(x) = (x_1 - 2)^4 + (x_1 - 2x_2)^2$ 
is the objective function of problem HS24 in CUTEst collection,   and  $\varepsilon_i$ values are drawn from Gaussian distribution $\mathcal{N}(0,\sigma^2).$ Different levels of noise, i.e., variance, are employed to model different levels of similarity of local cost functions, where higher level of noise indicates more heterogeneous data.

\begin{figure}[H]
    \centering
    \includegraphics[width=\textwidth]{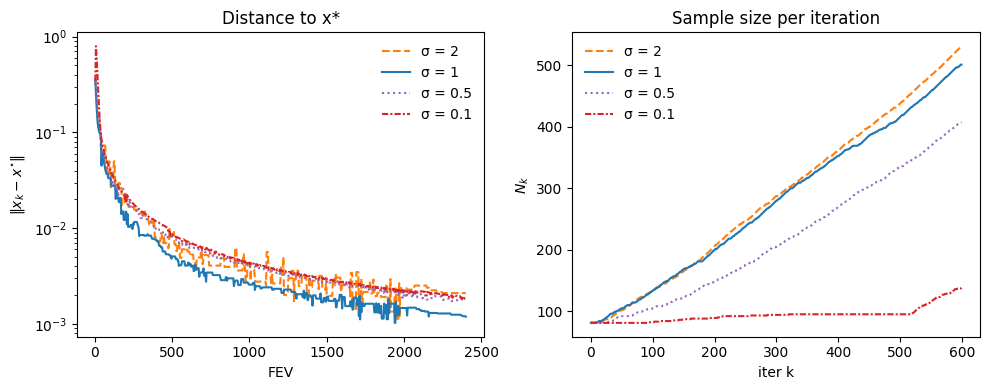}
     \makebox[\textwidth]{%
        \begin{minipage}[t]{0.48\textwidth}
            \centering
            \footnotesize a)
        \end{minipage}%
        \begin{minipage}[t]{0.48\textwidth}
            \centering
            \footnotesize b)
        \end{minipage}%
       
    }
    \caption{\footnotesize{Influence of dissimilarity of local cost functions (modeled by $\sigma \in \{0.1, 0.5, 1, 2\}$)  on  the behavior of ASPEN algorithm applied on problem \eqref{problem52}: a) optimality gap; b) sample size increase. }}
     \label{noise}
\end{figure} 

The results presented in Figure \ref{noise} show the robustness of \text{ASPEN} with respect to the optimality gap and illustrate the adaptive nature of the proposed method. As expected, more heterogeneous data require larger mini-batch sizes to mimic the original objective function properly, and these results clearly indicate that ASPEN is adaptable with respect to the sample size increase. 

\section{Conclusions}\label{sec6}

 We introduced a novel first-order adaptive sampling algorithm for finite-sum minimization (ASPEN) that extends the work of \cite{krejic5} to a more general class of problems with nonlinear equality constraints. The method combines an additional sampling technique with non-monotone line search and puts it into the framework of quadratic penalty methods. The resulting method may behave like a mini-batch or an increasing sample method, depending on the problem at hand. Besides the sample size, the penalty parameter is also updated in an adaptive manner. 
 We proved almost sure convergence of ASPEN  under some standard assumptions for the considered framework, thus providing theoretical support for the proposed method. Numerical results conducted on real-world binary classification problems show that ASPEN is competitive with other state-of-the-art methods. Moreover, a numerical study on an academic problem reveals ASPEN's capability of adapting to different data structures. 
Future work will include potential extensions to problems with nonlinear inequality constraints. 

\vspace{1cm}

%{\bf{Acknowledgement.}   }...
%We are grateful to the associate editor and two anonymous referees whose comments helped us improve the paper.}  

{\bf{Funding.}  } 
N. Kreji\' c and N. Krklec Jerinki\' c are supported by the Science Fund of the Republic of Serbia, Grant no. 7359, Project LASCADO.  T. Ostoji\' c is partially supported by the Ministry of Science, Technological Development and Innovation of the Republic of Serbia (Grants No. 451-03-136/2025-03/200156) and by the Faculty of Technical Sciences, University of Novi Sad through project “Scientific and Artistic Research Work of Researchers in Teaching and Associate Positions at the Faculty of Technical Sciences, University of Novi Sad 2025” (No. 01-50/295). N. Vu\v ci\' cevi\' c is partially supported by the Ministry of Science, Technological Development and Innovation of the Republic of Serbia (Grant no. 451-03-137/2025-03/ 200122)

%{\bf{Availability statement.}  } The datasets analyzed during the current study are available in ...
%the MNIST database of handwritten digits \cite{MNIST},  LIBSVM Data: Classification (Binary Class) \cite{SPLiADL} and UCI Machine Learning Repository \cite{MUSH}.
%All data generated or analyzed during this study are included in this published article.\\

%{\bf{Disclosure statement}}

%{\bf {Conflict of interest.}  } The authors declare no competing interests.


\begin{thebibliography}{99}
\bibitem{berahas1}
{\sc A. S. Berahas,  R. Bollapragada, \&  B. Zhou}, (2022). 
An adaptive sampling sequential quadratic programming method for equality constrained stochastic optimization. %Available at: {\em arXiv:2206.00712.}  

\bibitem{berahas3} 
{\sc A. S. Berahas,   F. E. Curtis,  M. J. O’Neill, \&  D. P. Robinson}, (2024). A stochastic sequential quadratic optimization algorithm for nonlinear-equality-constrained optimization with rank-deficient Jacobians. {\em Mathematics of Operations Research, 49(4), 2212-2248.}  

\bibitem{berahas2} 
{\sc A. S. Berahas,  F. E. Curtis,  D. Robinson, \&  B. Zhou}, (2021). Sequential quadratic optimization for nonlinear equality constrained stochastic optimization. {\em SIAM Journal on Optimization, 31(2), 1352-1379.} 

\bibitem{Numerika}{\sc A. S. Berahas, J. Shi, Z. Yi, B. Zhou,} (2023).
 Accelerating stochastic sequential quadratic programming for equality constrained optimization using predictive variance reduction,
 {\em Computational Optimization and Applications, 86(1), 79-116.}  
 
 \bibitem{libsvm}{\sc C. C. Chang, C. J. Lin,} (2011).
LIBSVM: A library for support vector machines,
{\em ACM Transactions on Intelligent Systems and Technology, 2:27:1--27:27.}  


\bibitem{curtis} 
{\sc F. E. Curtis,   D. P. Robinson, \&  B. Zhou,} (2021). Inexact sequential quadratic optimization for minimizing a stochastic objective function subject to deterministic nonlinear equality constraintsi. Available at: {\em arXiv: 2107.03512.}  

\bibitem{curtis2}
 {\sc F. E. Curtis,  M. J. O’Neill, \&  D. P. Robinson}, (2024). Worst-case complexity of an SQP method for nonlinear equality constrained stochastic optimization. {\em Mathematical Programming, 205(1), 431-483.}   

\bibitem{LSOS}{\sc D. Di Serafino, N. Krejić, N. Krklec Jerinkić, M. Viola,} (2023).
 LSOS: Line-search Second-Order Stochastic optimization methods for nonconvex finite sums,  
 {\em Mathematics of Computation, 92(341), 1273-1299.}  


\bibitem{fang} 
{\sc  Y. Fang,  S. Na,  M. W. Mahoney, \&  M. Kolar}, (2024). Fully stochastic trust-region sequential quadratic programming for equality-constrained optimization problems. {\em SIAM Journal on Optimization, 34(2), 2007-2037.}  

\bibitem{cutest}
{\sc N. I. M. Gould, D. Orban, \& P. L. Toint}, (2015).
CUTEst: a Constrained and Unconstrained Testing Environment with safe threads for mathematical optimization.
{\em Computational Optimization and Applications, 60, 545–557.}

%\bibitem{grippo}
%{\sc L. Grippo, F. Lampariello, \& S. Lucidi,} (1986).
%A nonmonotone line search technique for Newton's method,
%{\em SIAM Journal on Numerical Analysis 23(4), pp. 707-716.}
 
\bibitem{huyer} 
{\sc  W. Huyer, \&  A. Neumaier,} (2003). A new exact penalty function. {\em SIAM Journal on Optimization, 13(4), 1141-1158.}  

\bibitem{lifukushima}{\sc D. H. Li, \& M. Fukushima,} (2000). 
A derivative-free line search and global convergence of Broyden-like method for nonlinear equations. 
{\em Optimization Methods and Software, 13(3), 181–201.}
%https://doi.org/10.1080/10556780008805782

\bibitem{dataset1} {\sc M. Lichman}, (2013). UCI Machine Learning Repository. 
Available at: {\em https://archive.ics.uci.edu/ml/index.php}


\bibitem{dataset2} {\sc Y. LeCun, C. Cortes, \& C. J. C. Burges}, (1998). 
The MNIST database of handwritten digits. Available at: {\em http://yann.lecun.com/exdb/mnist/}


\bibitem{iusem1}
{\sc  A. N. Iusem,  A. Jofré,  R. I. Oliveira, \&  P. Thompson,} (2019). Variance-based extragradient methods with line search for stochastic variational inequalities. 
{\em SIAM Journal on Optimization, 29(1), 175-206.} 
%{\color{blue} Ovo i naredna referenca su dva razlicita rada? Proverio sam - jesu dva razlicita rad}

\bibitem{iusem2}
{\sc  A. N. Iusem,  A. Jofré,  R. I. Oliveira, \&  P. Thompson,} (2017). 
Extragradient method with variance reduction for stochastic variational inequalities. 
{\em SIAM Journal on Optimization, 27(2), 686-724.}  

\bibitem{krejic1}
{\sc  N. Kreji\' c,  N. Krklec Jerinki\' c, \&  A. Ro\v znjik,} (2018).  
Variable sample size method for equality constrained optimization problems. 
{\em Optimization Letters, 12, 485-497.}  

\bibitem{krejic3} 
{\sc N. Kreji\' c \&  N. Krklec Jerinki\' c,} (2015). Nonmonotone line search methods with variable sample size. {\em Numerical Algorithms, 68(4), 711-739.}  

\bibitem{krejic4}
{\sc  N. Kreji\' c,  N. Krklec Jerinki\' c,  A. Martınez, \&  M. Yousefi,} (2024). 
A non-monotone trust-region method with noisy oracles and additional sampling.
{\em Computational Optimization and Applications, 89(1), 247-278.}

\bibitem{krejic5}
 {\sc N. Kreji\' c,  N. Krklec Jerinki\' c,  S. Rapaji\' c, \&  L. Rute\v si\' c,} (2025). 
 IPAS: An Adaptive Sample Size Method for Weighted Finite Sum Problems with Linear Equality Constraints. 
 Available at: {\em arXiv: 2504.19629.}  
 
 \bibitem{krejic2} 
{\sc  N. Kreji\' c,  Z. Lu\v zanin,  Z. Ovcin, \&  I. Stojkovska,} (2015). Descent direction method with line search for unconstrained optimization in noisy environment. {\em Optimization Methods and Software, 30(6), 1164-1184.}  

\bibitem{krklecjerinkic1}
{\sc  N. Krklec Jerinki\' c, \&  A. Ro\v znjik,} (2020). Penalty variable sample size method for solving optimization problems with equality constraints in a form of mathematical expectation. {\em Numerical Algorithms, 83(2), 701-718.}  
  
\bibitem{krklecjerinkic4}
{\sc N. Krklec Jerinkić, \& T. Ostojić,} (2024).
AN-SPS: adaptive sample size nonmonotone line search spectral projected subgradient method for convex constrained optimization problems, {\em Optimization Methods and Software, 39(5), 1143-1167.}  

\bibitem{23}
{\sc N. Krklec Jerinkić, F. Porta,  V. Ruggiero, \& I. Trombini,} (2025).
Variable metric proximal stochastic gradient methods with additional sampling, {\em Computational Optimization and Applications, https://doi.org/10.1007/s10589-025-00720-w.} %to appear.% Available at: {\em Optimization Online: https://optimization-online.org/?p=29153. {\color{red}Mozemo ostaviti svakako link?}} 

\bibitem{lsnmbb} {\sc N. Krklec Jerinkić, V. Ruggiero, \& I. Trombini,} (2025). Spectral Stochastic Gradient Method with Additional Sampling for Finite and Infinite Sums, {\em Computational Optimization and Applications, 91 (2), 717–758.} 
%{\color{blue} Dopuniti referencu} Ovako je prema google scholar-u citaranje.


\bibitem{lan}
{\sc  G. Lan,} (2020). First-order and stochastic optimization methods for machine learning, {\em  Cham: Springer International Publishing, (Vol. 1), 21-51.  }
%{\color{blue} Nesto nedostaje?} Dodala sam Vol.1. Ovako je prema google scholar-u citaranje.


\bibitem{na1}
{\sc  S. Na,  M. Anitescu, \&  M. Kolar,} (2023). An adaptive stochastic sequential quadratic programming with differentiable exact augmented lagrangians. {\em Mathematical Programming, 199(1), 721-791.}  

\bibitem{na2} 
{\sc  S. Na,  M. Anitescu, \&  M. Kolar,} (2023). Inequality constrained stochastic nonlinear optimization via active-set sequential quadratic programming. {\em Mathematical Programming, 202(1), 279-353.} 

\bibitem{nandwani}
 {\sc  Y. Nandwani,  A. Pathak, \& P. Singla,}(2019).  
 A primal dual formulation for deep learning with constraints. {\em Advances in neural information processing systems, 32, 12157-12168.}  %{\color{blue} Nije kompletna referenca} 


\bibitem{negiar}
{\sc  G. Négiar,  G. Dresdner,  A. Tsai,  L. El Ghaoui,  F. Locatello,  R. Freund, \&  F. Pedregosa,} (2020, November). 
Stochastic Frank-Wolfe for constrained finite-sum minimization. {\em In international conference on machine learning, PMLR, 7253-7262.}  %{\color{blue} Nesto nedostaje?} 

\bibitem{polak} 
{\sc  E. Polak, \&  J. O. Royset,} (2008). Efficient sample sizes in stochastic nonlinear programming. {\em Journal of Computational and Applied Mathematics, 217(2), 301-310.}   



\bibitem{wang} 
{\sc  X. Wang, S. Ma,  \&  Y. X. Yuan,} (2017). Penalty methods with stochastic approximation for stochastic nonlinear programming. {\em Mathematics of Computation, 86(306), 1793-1820.}   



\end{thebibliography}
\end{document}